\documentclass[a4paper,reqno, 11pt]{amsart}

\usepackage[utf8]{inputenc}
\usepackage{tikz}
\usepackage{tikz-cd}
\usepackage[top=2.3cm,bottom=2.3cm,left=2.3cm,right=2.3cm]{geometry}
\usepackage{graphicx}
\usepackage{mathtools}
\usepackage{hyperref}

\usepackage[dvips]{epsfig}
\usepackage{color}
\usepackage{verbatim}
\usepackage{amsgen, amstext,amsbsy,amsopn,amssymb, amsthm, mathptmx, amsfonts,amssymb,amscd,amsmath,nicefrac,euscript,enumerate,url,verbatim,calc,framed}
\usepackage[all]{xy}
\def\red#1{{\color{red}#1}}
\def\blue#1{{\color{blue}#1}}

\theoremstyle{plain}
\newtheorem{thm}{\bf Theorem}[section]

\newtheorem{prop}[thm]{\bf Proposition}
\newtheorem{lem}[thm]{\bf Lemma}
\newtheorem{cor}[thm]{\bf Corollary}

\theoremstyle{definition}
\newtheorem{defn}[thm]{\bf Definition}

\theoremstyle{remark}
\newtheorem{rem}[thm]{\bf Remark}

\newtheorem{note}[thm]{\bf Note}
\newtheorem{obs}[thm]{\bf Observation}

\DeclareMathOperator{\reg}{reg}
\DeclareMathOperator{\Tor}{Tor}

\def\NN{\mathbb{N}}
\def\ZZ{\mathbb{Z}}

\newsavebox\foobox

\begin{document}

\title[Some classes of sequences of linear type]{Some classes of sequences of linear type}


\author[Neeraj Kumar]{Neeraj Kumar}
\address{Department of Mathematics, Indian Institute of Technology Hyderabad, Kandi, Sangareddy - 502285}
\email{neeraj@math.iith.ac.in}


\author[Chitra Venugopal]{Chitra Venugopal}
\address{Department of Mathematics, Indian Institute of Technology Hyderabad, Kandi, Sangareddy - 502285}
\email{ma19resch11002@iith.ac.in}




\date{\today}

\begin{abstract}
Given a graded ring \( A \) and a homogeneous ideal \( I \), the ideal is said to be of linear type if the Rees algebra of \( I \) is isomorphic to the symmetric algebra of \( I \). In general, $y$-regularity of Rees algebra of $I$ is  $0 \Rightarrow$ $I$ is generated by a $d$-sequence $\Rightarrow I$ is of linear type. We show that \( d \)-sequence ideals represent a significantly smaller subset of ideals of linear type in terms of \( y \)-regularity. Moreover, we identify a class of \( d \)-sequences whose arbitrary powers generate ideals of Gr\"obner linear type.  Notably, while \( d \)-sequences are inherently weak \( d \)-sequences, we highlight a specific class of algebras where weak \( d \)-sequences are indeed \( d \)-sequences.
\end{abstract}

\begin{center}
\maketitle
\end{center}

\date{\today}

\section*{Introduction}

Let \( A \) be a commutative graded ring and \( I \) be a homogeneous ideal of \( A \). Consider \( \text{Sym}(I) = \oplus_{i \geq 0} S^i(I) \) to be the symmetric algebra of \( I \), where \( S^i(I) \) denotes the \( i \)-th symmetric power of \( I \), and \( \mathcal{R}(I) = \oplus_{i \geq 0} I^i \) to be the Rees algebra of \( I \). In general, there exists a canonical surjection, from \( \text{Sym}(I) \) to \( \mathcal{R}(I) \). An ideal \( I \) is said to be of linear type if the Rees algebra of \( I \) coincides with the symmetric algebra of \( I \). The term `linear type' derives from the fact that, when $\text{Sym}(I) \cong \mathcal{R}(I)$, the defining relations of the Rees algebra are linear in the new variables introduced in the Rees algebra construction. The presentation matrices of these ideals contribute to the defining relations of the corresponding Rees algebras, and hence a significant amount of research has been and continues to be done on these classes of ideals.

\vspace{1mm}

Different types of sequences have been defined over time, which aids in the study of Rees and symmetric algebras. There is, therefore, a plethora of literature in this area. We are particularly interested in the types of sequences that generate ideals of linear type, specifically \(d\)-sequences.

\vspace{1mm}

The theory of \(d\)-sequences was introduced by Huneke in the 1980s \cite{CH1,CH2} as a notion of a "weak" regular sequence, to help in the study of the depth of asymptotic powers of a homogeneous ideal in a graded ring. Numerous examples of \(d\)-sequences are provided in \cite{CH1}. Recently, combinatorial characterizations for edge binomials of trees and unicycle graphs forming \(d\)-sequences have been given in \cite{AN2, AN1}.

\vspace{1mm}

Huneke (\cite{CH2}) and Valla (\cite{Valla}) independently proved that \(d\)-sequences generate ideals of linear type. By the work of Römer in \cite{Rom1}, the degree of the syzygies of the Rees algebras of the ideals generated by \(d\)-sequences is bounded above by the homological degree. For the general class of ideals of linear type, it is evident that the first syzygies of the Rees algebras of these ideals are linear in the new set of variables. However, it would be interesting to determine if there is a bound on the degrees of the higher syzygies in general. In one of the main results in this article, Theorem \ref{LT_no-D}, we show that there is no such bound on the degree of the syzygies for the general class of ideals of linear type. 

\vspace{1mm}

Conca, Herzog, and Valla \cite{CHV} proved that an ideal \( I \) is of (Gröbner) linear type if its initial ideal with respect to some monomial order \( \tau \) is of (Gröbner) linear type. It has been observed that properties such as Cohen-Macaulayness and normality are also preserved when moving from \( \mathcal{R}(I) \) to \( \mathcal{R}(\text{in}_{\tau}(I)) \). This observation has motivated research into conditions that guarantee a monomial ideal is of linear type. A homogeneous ideal \( I \) of a standard graded ring \( A \) is of Gröbner linear type if it is of linear type, and the linear relations of the defining ideal of \( \mathcal{R}(I) \cong A[Y]/J \) form a Gröbner basis with respect to some monomial order on \( A[Y] \). An important class of monomial sequences is \( M \)-sequences, introduced by Conca and De Negri in \cite{CN}. Besides being of Gröbner linear type, the explicit defining relations of the Rees algebra of ideals generated by such sequences are also known. In Section \ref{$M$-sequences and $d$-sequences}, we describe conditions under which an \( M \)-sequence becomes a \( d \)-sequence and vice versa.

\vspace{1mm}

The concept of a weak \(d\)-sequence first appeared in \cite{CH4} and was defined by Huneke, where the ordering of the elements in the sequence is with respect to a finite partially ordered set. These sequences help in the computation of the depth of the asymptotic powers of the ideals generated by them. There are many natural examples of weak \(d\)-sequences in \cite{CH1, CH4}. It is known that \(d\)-sequences are weak \(d\)-sequences \cite[Corollary 1.1]{CH4}.

\vspace{1mm}

Another major class of ideals generated by weak \(d\)-sequences comes from algebras with straightening laws defined by DeConcini et al. 
\cite{Hodge_algebras}. Maximal order Pfaffians corresponding to a generic skew-symmetric matrix of odd order are an example of such a class of ideals \cite[Example 1.20]{CH4}. For a skew-symmetric matrix \( X \), the Pfaffian of \( X \), denoted by \(\text{Pf}(X)\), is defined as the square root of the determinant of \( X \) \cite{Artin}. Maximal order Pfaffians of a skew-symmetric matrix of odd order \( n \) are obtained by considering the Pfaffians of submatrices of order \( n-1 \) obtained by deleting a row and the corresponding column of the matrix \( X \) \cite{DBDE}. An ideal generated by these Pfaffians corresponding to a skew-symmetric matrix of indeterminates and odd order is proved to be of linear type \cite{Baetica}. In Corollary \ref{max-pfaffians}, we prove that these Pfaffians, in fact, form a \(d\)-sequence.

\section{Preliminaries} \label{sec1}

\begin{defn} \label{first_defn_dseq}
Let $\{ \bf{a} \}$ denote a sequence $\{ a_1, \ldots, a_n \}$ in a graded ring $A$ and $N$ be a finitely generated graded $A$-module. Let $I=\langle a_1, \ldots, a_n\rangle$ and $I_i=\langle a_1, \ldots, a_i \rangle$. Then the sequence $\{\bf{a}\}$ is a \textit{$d$-sequence} for $N$ if for each integer $i=1, \ldots, n$ and each integer $k=i, \ldots, n$,
$$(I_{i-1}N:_N a_i) \cap I = I_{i-1}N$$
and $\{\bf{a}\}$ minimally generates $I$.
\end{defn}

\begin{thm} \cite[Theorem 3.1]{CH2} \label{d-seq_LT}
Let $A$ be a ring and $\{a_1, \ldots, a_n \}$ be  a $d$-sequence in $A$. Then $I=\langle a_1, \ldots, a_n \rangle$ is of linear type.  \end{thm}

Proper sequences and $s$-sequences were introduced to study invariants associated with symmetric algebras \cite{HSV1,s-seq}

   \begin{defn}
   For a finitely generated $A$-module $N$, its generating set $\{a_1, \ldots, a_n\}$ form an \textit{$s$-sequence} (with respect to an admissible term order for the monomials in $y_i$ with $y_1 <y_2 <\ldots <y_n$) if for $Sym(N) \cong S/J$ where $S$ is a bigraded polynomial ring $S=A[y_1,\ldots,y_n]$ and $J$ the defining ideal of the symmetric algebra, $in(J)= \langle L_1y_1,  \ldots, L_ny_n  \rangle $ where $L_i=(\langle a_1,\ldots,a_{i-1} \rangle N:a_i)$. 
   
   In particular, if $L_1 \subseteq L_2 \subseteq \ldots \subseteq L_n$, then $\{a_1, \ldots, a_n\}$ is said to form a \textit{strong $s$-sequence} (cf. \cite{s-seq}).
   \end{defn}

\begin{defn} \label{bigraded_reg}
Let $R$ be a bigraded $K$-algebra where $K$ is a field and $M$ be a finitely generated $R$-module.
Let $t_{i_y}^R(M)=\sup\{j: \Tor_i^R(M,K)_{(*,j)} \neq 0\}$ with  $t_{i_y}^R(M)=- \infty$ if $\Tor_i^R(M,K)_{(*,j)}=0$ for all $j \geq 0$. Then the \textit{$y$-regularity} of $M$ denoted by  reg$_y^R(M)$ is then defined as, 
$$ \reg_y^R(M)=\sup \{ t_{i_y}^R(M) -i, \, i \geq 0\}.$$
\end{defn}

For a homogeneous ideal $I$ of a graded ring $A$, $\mathcal{R}(I)\cong \oplus_{i \geq 0}I^i$ can be seen as a bigraded algebra with the natural bigrading given by $\mathcal{R}(I)_{(i,j)}=(I^j)_i$. Similarly, symmetric algebras can be seen as bigarded algebras. The following result by Romer characterizes an ideal generated by a strong $s$-sequence and a $d$-sequence in terms of the vanishing of $y$-regularity of the corresponding symmetric algebra and Rees algebra, respectively. 

\begin{thm} (\cite[Corollary 3.2 \footnote{In this result, by an $s$-sequence, the author means a strong $s$-sequence.}]{Rom1} ) \label{Romer_characterization}
 Let $I=\langle a_1, \ldots ,a_m \rangle \subset K[x_1, \ldots, x_n]$ be an equigenerated graded ideal. Then, 
 \begin{enumerate}
     \item  $I$ is generated by an $s$-sequence (with respect to the reverse lexicographic order) if and only if reg$_y(Sym(I))=0$.
     \item $I$ is generated by a $d$-sequence if and only if reg$_y(\mathcal{R}(I))=0$.
 \end{enumerate}
\end{thm}

\begin{defn}
A sequence $\{ \bf{a} \}$ is a \textit{proper sequence} if $a_iH_j(a_1, \ldots, a_{i-1})=0$ for $i=1, \ldots, n$ and $j \geq 1$ where $H_j(a_1, \ldots, a_{i-1})$ denotes the $j^{th}$ homology module of the Koszul complex on $ \{ a_1, \ldots, a_n \}$. (cf. \cite{HSV1})
\end{defn}

The following result shows that the notion of strong $s$-sequences is equivalent to the notion of proper sequences.

\begin{prop} \cite[Corollary 3.4]{s-seq} \label{strong_s_seq&proper_seq}
Let $I$ be an ideal generated by a sequence $\{\bf{a} \}$ $=\{a_1, \ldots, a_n \}$ in a ring. Then $\{ \bf{a} \}$ is a strong $s$-sequence with respect to the reverse lexicographic term order if and only if $\{\bf{a} \}$ is a proper sequence. 
\end{prop}

Tang gave equivalent conditions for a monomial sequence to be a $d$-sequence or a proper sequence. For $i, j \in \NN$, let $(m_i,m_j)$ denote the greatest common divisor of $m_i$ and $m_j$.

\begin{thm} \label{monomial_d&proper_seq}
Let $\{a_1, \ldots, a_n\}$ be a monomial sequence. Then,
\begin{enumerate}
    \item \cite[Theorem 3.1]{monomial_seq} $\{a_1, \ldots, a_n\}$ is a proper sequence if and only if  $m_i \nmid m_j$ for $i \neq j$ and further satisfy the condition $(m_i,m_j)|m_k$ for $1 \leq i <j <k \leq n$. 
    \item \cite[Theorem 2.1]{monomial_seq}  $\{a_1, \ldots, a_n\}$ is a $d$-sequence if and only if it satisfies the condition of a proper sequence and further satisfies the condition $(m_i,m_j)=(m_i,m_j^2)$ for $ 1 \leq i <j\leq s$.
\end{enumerate}
\end{thm}

$M$-sequences are a special type of monomial sequence that have certain desirable properties associated with them.

\begin{defn}
A sequence of monomials $\{ m_1, \ldots, m_s \}$ in a set of indeterminates $X=\begin{bmatrix} x_1 & \cdots & x_n \end{bmatrix}$ is said to be an \textit{$M$-sequence}, if for all $1 \leq i \leq s$, there exists a total order on the set of indeterminates, say $x_1 \prec \ldots \prec x_n$ with $m_i=x_1^{a_{i_1}} \cdots x_n^{a_{i_n}}$ and $a_{i_1} >0, \ldots, a_{i_n}>0$, such that whenever $x_k | m_j$ with $1 \leq k \leq n$ and $i<j$, then $x_k^{a_{i_k}}\cdots x_n^{a_{i_n}}|m_j$. (cf. \cite{CN})
\end{defn}

Some interesting properties of $M$-sequences are the following:

\begin{thm} \cite[Theorem 2.4]{CN} \label{m-seq_glt}
Let $I$ be an ideal generated by an $M$-sequence $\{m_1, \ldots, m_s\}$ in a set of indeterminates $X$ over a field $K$. Let $\mathcal{R}(I) \cong S/J$, where $S = K[X, Y]$, $Y = \begin{bmatrix} y_1 & \cdots & y_s \end{bmatrix}$, and $J$ is the defining ideal of the Rees algebra of $I$. Then the minimal generators of $J$ have the form $\dfrac{m_i}{(m_i, m_j)} y_j - \dfrac{m_j}{(m_i, m_j)} y_i$
for $1 \leq i < j \leq s$, and they form a Gr\"obner basis of $J$
\end{thm}

\begin{lem} \cite[Lemma 2.2]{CN} \label{powers_of_mseq}
Let $\{m_1, \ldots, m_s \}$ be an $M$-sequence in a set of indeterminates $X$. Then $\{m_1^{n_1}, \ldots, m_s^{n_s} \}$ is an $M$-sequence where $1 \leq n_1 \leq \cdots \leq n_s$ are integers.
\end{lem}

\section{A class of ideals of linear type, not generated by d-sequences}  \label{sec2}

In this section, we provide a class of ideals of linear type with an increasing value of $y$-regularity of the associated Rees algebras. In other words, we show how irregularly the regularity of the Rees algebra corresponding to ideals of linear type can behave with respect to the second degree on moving away from the ideals generated by $d$-sequences.

For $n \in \NN$, let $C_n$ be a cycle on $n$ vertices labelled as $x_1, \ldots,x_n$, $P_{\ell}(C_n)$ denote the ideal generated by the paths of length $\ell$ in $C_n$ in $B=K[x_1,\ldots, x_n]$. It is observed that when $n$ is odd ($n \geq 5$),  $I=P_{n-3}(C_n)=\langle m_1, \ldots, m_n \rangle$, where $m_i=x_ix_{i+1} \cdots x_{i+n-3}$ with indices in $\ZZ_n$, gives a class of ideals of linear type with an increasing value of $y$-regularity of $\mathcal{R}(I)$.

From the results in \cite{Brumatti_Silva}, \cite{Ali_Sara} and \cite{Erey}, the minimal free resolution of $B/I$ has the form,
%

{\small{
  \begin{equation}\label{resolution}
    0 \longrightarrow B(-n) \stackrel {\phi_{3}} \longrightarrow B(-(n-1))^{n} \stackrel {\phi_{2}}
    \longrightarrow B(-(n-2))^{n}\stackrel {\phi_{1}} \longrightarrow B \longrightarrow 0
    \end{equation}

where $\phi_1= \begin{bmatrix}
   m_1 &\cdots & m_{n} \\
   \end{bmatrix}$, $\phi_2= \begin{bmatrix}
   x_{n-1} & 0 & 0 & \cdots &0 & -x_{n} \\
   -x_1 & x_n & 0 & \cdots & 0 &0 \\
   0 & -x_2 & x_1 & \cdots &0 &0 \\
   \vdots & \vdots & \vdots & \ddots & \vdots & \vdots \\
   0 & 0 & 0 & \cdots & x_{n-3} & 0 \\
   0 & 0 &0 & \cdots & -x_{n-1} &x_{n-2}
   \end{bmatrix}$
   and 

  $\phi_3= \begin{bmatrix} 
   x_n & x_1 & x_2 & \cdots & x_{n-1}
   \end{bmatrix}^T$}}.

Further it has been proved in \cite[Theorem 3.4]{Brumatti_Silva} that $I=P_{n-3}(C_n)$ is of linear type. However, the following result shows that it, in fact, produces a class of ideals with an increasing value of $y$-regularity and hence not generated by $d$-sequences.    

\begin{thm} \label{LT_no-D}
       Let $n=2r+1$, $r \geq 2$ and $C_n$ be a cycle on $n$ vertices. Then for $I=P_{n-3}(C_n)$, reg$_y\mathcal{R}(I) \geq r-1$.
\end{thm}

\begin{proof}
Let $I=P_{n-2}(C_n)$ for $n=2r+1$, $r \geq 2$. Then the defining relations of the Rees algebra of $I$ is given by $\phi_2 \cdot Y^T \subseteq K[X,Y]:=S$ where $X=\begin{bmatrix} x_1 \cdots x_n \end{bmatrix}$ and $Y=\begin{bmatrix} y_1 \cdots y_n \end{bmatrix}$ with bidegree of $x_i=(1,0)$ and bidegree of $y_j=(0,1)$, $1 \leq i,j \leq n$. In particular, the defining relations can be explicitly given by 
\begin{equation} \label{def_eqns}
x_{i-2}y_i-x_{i}y_{i+1}, \, 1 \leq i \leq n
\end{equation}
where the indices are considered in $\ZZ_n$.

Let $\delta_1^n:S^n \rightarrow S$ denote the map in the minimal free resolution of $\mathcal{R}(I)$ where $I=P_{n-3}(C_n)$, defined as $\delta_1^n(e_i)=x_{i-2}y_i-x_{i}y_{i+1}$ where $e_i$, $1 \leq i\leq n$ is a basis of $S^n$ and the indices are in $\mathbf{\ZZ}_n$. 


For $n=5$, {\footnotesize$$\delta_1^5(e_1)=x_4y_1-x_1y_2,\; \delta_1^5(e_2)=x_5y_2-x_2y_3, \; \delta_1^5(e_3)=x_1y_3-x_3y_4, \; \delta_1^5(e_4)=x_2y_4-x_4y_5, \; \delta_1^5(e_5)=x_3y_5-x_5y_1$$}
Then substituting the above values one obtains, 
\begin{equation} \label{n=5}
\delta_1^5(y_3y_5e_1+y_1y_4e_2+y_2y_5e_3+y_1y_3e_4+y_2y_4e_5)=0.
\end{equation}

We claim that for $n=2r+1$, $ r \geq 3$, the following relation belongs to the kernel of $\delta_1^n$.
\begin{equation} \label{eq_yreg}
\sum_{i=1}^{n-2} (\alpha_{i \, n-2}ye_i) + y_1y_3\cdots y_{2j+1} \cdots y_{n-2}e_{n-1}+y_2y_4 \cdots y_{2j} \cdots y_{n-1}e_{n}
\end{equation} 
where $j \in \NN$ and $y= \left \{ \begin{array}{ccc}  
y_n & if  & i \equiv 1 \mod 2 \\
y_{n-1} & if  & i \equiv 0 \mod 2
\end{array}
\right.$, $\alpha_{i \, n-2}$ denotes the coefficients of $e_i$ for $1 \leq i \leq n-2$ of bidegree $(0,r-1)$ in the kernel of $\delta_1^{n-2}$ of the form of relation (\ref{eq_yreg}).  

We prove this by applying induction on \( r \). For \( r = 3 \), by computations similar to the \( n = 5 \) case, we obtain:
\[
\delta_1^7 (y_3y_5y_7e_1 + y_1y_4y_6e_2 + y_2y_5y_7e_3 + y_1y_3y_6e_4 + y_2y_4y_7e_5 + y_1y_3y_5e_6 + y_2y_4y_6e_7) = 0
\]

This expression is of the required form, with:
\[
\alpha_{15} = y_3y_5, \quad \alpha_{25} = y_1y_4, \quad \alpha_{35} = y_2y_5, \quad \alpha_{45} = y_1y_3, \quad \alpha_{55} = y_2y_4
\]
coming from Equation (\ref{n=5})


Now assume $n=2r+1$, $r > 3$. Then by induction hypothesis, $$ \sum_{i=1}^{n-4} (\blue{\alpha_{i \, n-4}z_i}e_i) + \blue{y_1y_3\cdots y_{2j+1} \cdots y_{n-4}}e_{n-3}+\blue{y_2y_4 \cdots y_{2j} \cdots y_{n-3}}e_{n-2}$$ belongs to the kernel of $\delta_1^{n-2}$ where $j \in \NN$ and $z_i= \left \{ \begin{array}{ccc}  
y_{n-2} & if  & i \equiv 1 \mod 2 \\
y_{n-3} & if  & i \equiv 0 \mod 2
\end{array}
\right..$ 

\vspace{2mm}

Since $\delta_1^n(e_i)=x_{i-2}y_i-x_{i}y_{i+1}$ for $ 1 \leq i \leq n$, one obtains, 
\[
\begin{split}
& \sum_{i=1}^{n-4} (\blue{\alpha_{i \, n-4}z_i}w_i\delta_1^n(e_i)) + \blue{y_1y_3\cdots y_{2j+1} \cdots y_{n-4}}y_{n-1}\delta_1^n(e_{n-3})+\blue{y_2y_4 \cdots y_{2j} \cdots y_{n-3}}y_{n}\delta_1^n(e_{n-2})  \\
& + y_1y_3\cdots y_{2j+1} \cdots y_{n-2}\delta_1^n(e_{n-1})+y_2y_4 \cdots y_{2j} \cdots y_{n-1}\delta_1^n(e_{n})=0,
\end{split}
\]
where $j \in \NN$ and $w_i= \left \{ \begin{array}{ccc}  
y_{n} & if  & i \equiv 1 \mod 2 \\
y_{n-1} & if  & i \equiv 0 \mod 2
\end{array}
\right.$ This implies that relation (\ref{eq_yreg}) belongs to the kernel of $\delta_1^n$.

To prove that $y$-regularity strictly increases for this class of ideals, it suffices to show that relation (\ref{eq_yreg}) cannot be generated by elements in the kernel of $\delta_1^n$ with $y$-degree strictly less than $r$.

To understand the proof, consider the case \( I = P_4(C_7) \). Let:
\begin{equation} \label{eg}
\delta_1^7 \left( \sum_{i=1}^{7} m_i e_i \right) = 0
\end{equation}
for some homogeneous polynomials \( m_i \in S \) of bidegree \( (0,2) \). Now, consider \( \sum_{i=1}^{7} m_i \delta_1^7(e_i) \) as polynomials in \( A[X] \), where \( A \) is the field of fractions of \( K[Y] \).

Since \( \delta_1^7(e_3) = x_1 y_3 - x_3 y_4 \) and \( \delta_1^7(e_5) = x_3 y_5 - x_5 y_6 \), following the coefficients of \( x_1 \) and \( x_3 \) in these equations, one finds that for the coefficient of \( x_1 \) to vanish in Equation (\ref{eg}), \( y_3 y_5 \) must divide \( m_1 \). However, since \( m_1 \) is of bidegree \( (0,2) \), this implies \( m_1 = y_3 y_5 \). But, since \( \delta_1^7(e_7) = x_5 y_7 - x_7 y_1 \), the coefficient of \( x_5 \) in \( \delta_1^7(e_7) \) will be \( y_7 \), and clearly \( y_7 \nmid m_1 \). This is a contradiction to Equation (\ref{eg}).

We give a general proof for the same in the following paragraph.

Without loss of generality, assume that 
\begin{equation} \label{eg2}
\delta_1^n \left( \sum_{i=1}^{n} m_i e_i \right) = 0
\end{equation}
for some homogeneous polynomials \( m_i \in S \) of bidegree \( (0, r-1) \). Considering the general form of \( \delta_1^n(e_i) \) and viewing \( \sum_{i=1}^{n} m_i \delta_1^n(e_i) \) as a polynomial in \( S'' \), where \( S'' = K(Y)[X] \) (with \( K(Y) \) being the field of fractions of \( K[Y] \)), we find that \( y_3 y_5 \cdots y_{2(r-1)+1} \) must divide \( m_1 \) (by examining the coefficients of \( x_i \), where \( i = 2k+1 \), \( 0 \leq k \leq r-2 \), in \( \delta_1^n(e_i) \)).

However, since \( m_1 \) is of bidegree \( (0, r-1) \), it must be of the form \( m_1 = k y_3 y_5 \cdots y_{2(r-1)+1} \) for some \( k \in K \). But in \( \delta_1^n(e_{2r+1}) \), the coefficient of \( x_{2r-1} \) will be \( y_{2r+1} \), and clearly \( y_{2r+1} \nmid m_1 \). This implies that the coefficient of \( x_1 \) in \( \sum_{i=1}^{n} m_i \delta_1^n(e_i) \) would be non-zero, which is a contradiction to Equation (\ref{eg2}).

\end{proof}

\section{M-Sequences and d-Sequences} \label{$M$-sequences and $d$-sequences}

It is known that ideals generated by $M$-sequences are of Gröbner linear type, while ideals generated by $d$-sequences are of linear type. However, in general, there are no implications between $M$-sequences and monomial $d$-sequences. The following examples illustrate this:

\begin{enumerate}
\item An $M$-sequence need not be a $d$-sequence.\\
Let $B = K[x_1, x_2, x_3, x_4]$ and consider the sequence $\{x_1 x_3, x_3 x_4, x_2 x_4\}$. This sequence is an $M$-sequence but not a $d$-sequence, since $(\langle x_1 x_3 \rangle : x_3 x_4) \cap \langle x_1 x_3, x_3 x_4, x_2 x_4 \rangle = \langle x_1 x_3, x_1 x_2 x_4 \rangle \neq \langle x_1 x_3 \rangle$.

\item A monomial $d$-sequence need not be an $M$-sequence.\\
Let $B = K[x_1, x_2, x_3, x_4, x_5]$ and consider the sequence $\{x_1 x_2, x_3 x_4, x_1 x_5\}$. This sequence forms a $d$-sequence but is not an $M$-sequence.

\end{enumerate}

In this section, we provide conditions under which an $M$-sequence is a $d$-sequence and vice versa. Furthermore, we present some classes of $d$-sequences for which their arbitrary powers generate ideals of Gröbner linear type.\\

The following result gives conditions under which an $M$-sequence is a $d$-sequence.

\begin{prop} \label{m-seq}
Let $\{  \bf{a} \}$=$ \{ m_1, \ldots, m_s \}$ be a squarefree $M$-sequence in the set of indeterminates $X$. Then $\{  \bf{a} \}$ form a $d$-sequence if and only if for all $1 \leq j < s$ and $1 \leq k < j $ there exists $ 1 \leq l < j+1 $ such that $m_l (m_k,m_j)| m_k (m_l, m_{j+1})$ where $(m_i,m_j)$ denotes the greatest common divisor of $m_i$ and $m_j$.
\end{prop}

\begin{proof}

Let $\{ m_1, \ldots, m_s \}$ be a monomial sequence. In this case, the colon ideals can be specifically expressed as \[
(\langle m_1, \ldots, m_{j-1} \rangle : m_j) = \left\langle \frac{m_i}{(m_i, m_j)} \mid 1 \leq i < j \right\rangle.
\]
Since $\{ m_1, \ldots, m_s \}$ forms an $M$-sequence, the ideal $I$ generated by this sequence is of linear type.

Assume $\tau$ is the lexicographic term order on $S = K[X, Y]$, where $Y = \begin{bmatrix} y_1, \ldots, y_s \end{bmatrix}$, induced by the total order $y_s > y_{s-1} > \ldots > y_1$. Consequently, as a result of Theorem \ref{m-seq_glt}, the initial ideal of the defining relations of $\mathcal{R}(I)$ with respect to $\tau$ is generated by $\frac{m_i}{(m_i, m_j)} y_j$ for $1 \leq i < j \leq s$. 

The condition that for all $1 \leq j < s$ and $1 \leq k < j$, there exists $1 \leq l < j + 1$ such that $m_l (m_k, m_j) \mid m_k (m_l, m_{j+1})$ is equivalent to the inclusion 
\[
(\langle m_1, \ldots, m_{j-1} \rangle : m_j) \subseteq (\langle m_1, \ldots, m_j \rangle : m_{j+1}).
\]
Thus, $\{ m_1, \ldots, m_s \}$ forms a strong $s$-sequence equivalently a proper sequence (Proposition \ref{strong_s_seq&proper_seq}) if and only if the assumptions in the proposition are satisfied. Since $\{ m_1, \ldots, m_s \}$ is a square-free monomial sequence, it follows that $(m_i, m_j) = (m_i, m_j^2)$ for all $1 \leq i < j \leq s$. The result thus follows from Theorem \ref{monomial_d&proper_seq}.
\end{proof}

An important consequence of Proposition \ref{m-seq} is that it allows one to deduce information about the $y$-regularity of the Rees algebras of the ideals generated by such $M$-sequences.

\begin{cor}
Let $\{  \bf{a} \}$=$ \{ m_1, \ldots, m_s \}$ be a squarefree $M$-sequence in the set of indeterminates $X$ satisfying the assumptions of Proposition \ref{m-seq}. Then, $\reg_y(\mathcal{R}(I))=0$.
\end{cor}

\begin{proof}
The result follows from Proposition \ref{m-seq} and Theorem \ref{Romer_characterization}(2).
\end{proof}

It is known that $d$-sequences generate ideals of linear type. We give conditions for monomial $d$-sequences to generate ideals of Gr\"obner linear type.


\begin{prop} \label{powers_of_dseq}
Let $\{ \mathbf{a} \} = \{ m_1, \ldots, m_s \}$ be a monomial $d$-sequence in a set of indeterminates $X =\begin{bmatrix}
 x_1  & \cdots & x_n
\end{bmatrix}$. Assume there exists a total order on the set of indeterminates that appear in $m_i$, say $x_1 < \cdots < x_n$, with respect to which $m_i = x_{i_1}^{a_{i_1}} \cdots x_{i_l}^{a_{i_l}}$, where $a_{i_j} > 0$ for $j = 1, \ldots, l$, $1 \leq l \leq n$. Then $\{ \mathbf{a} \}$ is an $M$-sequence if it satisfies the following condition:
$$ \text{If } x_{i_k} \mid m_i \text{ where } 1 \leq k \leq {l} \leq n, \text{ then } x_{i_k}^{a_{i_k}} \cdots x_{i_{l}}^{a_{i_l}} \mid m_j \text{ where } j = \min\{ l: x_{i_k} \mid m_l, \; l > i \}. $$
\end{prop}

\begin{proof}
Let \( \{ \mathbf{a} \} = \{ m_1, \ldots, m_s \} \) be a monomial \( d \)-sequence satisfying the condition that if \( x_{i_k} \mid m_i \) where \( 1 \leq k \leq l \leq n \), then \( x_{i_k}^{a_{i_k}} \cdots x_{i_l}^{a_{i_l}} \mid m_j \) where \( j = \min\{ l : x_{i_k} \mid m_l, \; l > i \} \). Since \( \{ \mathbf{a} \} \) is a monomial \( d \)-sequence, from Theorem \ref{monomial_d&proper_seq}, \( (m_i, m_j) \mid m_k \) for \( 1 \leq i < j < k \leq n \). This implies that if \( x_{i_k} \mid m_j \) with \( 1 \leq k \leq n \) and \( i < j \), then \( x_{i_k}^{a_{i_k}} \cdots x_{i_l}^{a_{i_l}} \mid m_j \). Therefore, \( \{ \mathbf{a} \} \) satisfies the condition of an \( M \)-sequence.
\end{proof}

As a significant consequence of Proposition \ref{powers_of_dseq}, we obtain a class of $d$-sequences where the powers of these sequences generate ideals of Gr\"obner linear type.
\begin{cor} 
Let $\{ \mathbf{a} \} = \{ m_1, \ldots, m_s \}$ be a monomial $d$-sequence in a set of indeterminates $X$ satisfying the hypotheses of Proposition \ref{powers_of_dseq}. Then $\{ m_1^{n_1}, \ldots, m_s^{n_s} \}$ generate ideals of Gr\"obner linear type, where $1 \leq n_1 \leq \cdots \leq n_s$ are integers.
\end{cor}

\begin{proof}

Let $\{ \mathbf{a} \} = \{ m_1, \ldots, m_s \}$ be a monomial $d$-sequence in a set of indeterminates $X$ satisfying the hypotheses of Proposition \ref{powers_of_dseq}. Then $\{ \mathbf{a} \}$ forms an $M$-sequence. This implies, by Lemma \ref{powers_of_mseq}, that for $1 \leq n_1 \leq \cdots \leq n_s$, $\{ m_1^{n_1}, \ldots, m_s^{n_s} \}$ forms an $M$-sequence. Thus, $\{ m_1^{n_1}, \ldots, m_s^{n_s} \}$ generates an ideal of Gr\"obner linear type from Theorem \ref{m-seq_glt}.
\end{proof}

\section{A class of algebras where weak d-sequence implies d-sequence} \label{sec4}

Weak $d$-sequences, as the name suggests, are a concept of sequences weaker than $d$-sequences, which are defined with respect to finite partially ordered sets (posets) and are used to study the depths of powers of ideals of various determinantal varieties.

Let $(\Lambda, \leq)$ be a finite partially ordered set. A subset $\Sigma$ of $\Lambda$ is called a poset ideal if, for every $\alpha \in \Sigma$ and $\beta \leq \alpha$, we have $\beta \in \Sigma$. An element $\lambda \in \Lambda$ is said to lie above $\Sigma$ if $\lambda \notin \Sigma$ and, for every $\alpha \in \Lambda$, $\alpha < \lambda$ implies $\alpha \in \Sigma$. Let $\{ x_{\lambda} \mid \lambda \in \Lambda \}$ be a set of elements indexed by $\Lambda$ in a commutative ring $A$. Define $I = \langle x_{\lambda} \mid \lambda \in \Lambda \rangle$, the ideal generated by these elements. For each poset ideal $\Sigma \subseteq \Lambda$, let $I_{\Sigma} = \langle x_{\sigma} \mid \sigma \in \Sigma \rangle$, the ideal generated by the elements indexed by $\Sigma$. For an ideal $J$ of $A$, let $J^* = \langle x_{\beta} \mid x_{\beta} \in J \rangle$, the ideal generated by the elements $x_{\beta}$ that belong to $J$.

\begin{defn} (cf. \cite{CH4}) 
 Consider the notations mentioned above. A set $\{ x_{\lambda}| \lambda \in \Lambda \}$ of elements indexed by $\Lambda$ form a weak $d$-sequence with respect to $(\Lambda, \leq)$ if for each poset ideal $\Sigma $ of $\Lambda$ and each element $\lambda$ lying above $\Sigma$, the following holds. 
 \begin{enumerate}
     \item $(I_{\Sigma}:x_{\lambda})^*$ is generated by some set $\{ x_{\beta}| \beta \in \Sigma'  \}$ where $\Sigma' \subseteq \Lambda$ is a poset ideal.
     \item $(I_{\Sigma}:x_{\lambda}) \cap I =(I_{\Sigma}:x_{\lambda})^*.$
     \item If $x_{\beta} \in (I_{\Sigma}:x_{\lambda})$, then $x_{\lambda}x_{\beta} \in I_{\Sigma}I$.
     \item If $x_{\lambda} \notin (I_{\Sigma}:x_{\lambda})$, then $(I:x_{\lambda})=(I:x_{\lambda}^2)$.
 \end{enumerate}
\end{defn}

In general, it is established that a $d$-sequence is a weak $d$-sequence with respect to the total order \cite[Corollary 1.1]{CH4}. Thus, it is interesting to investigate the cases when a weak $d$-sequence becomes a $d$-sequence. In this section, we explore a class of algebras where weak $d$-sequences with respect to the total order are $d$-sequences.

For this purpose, we first define what is meant by an algebra with straightening law (ASL).

Let $A$ be a commutative $A'$-algebra where $A'$ is a ring, and let $(\Lambda, \leq)$ be a finite poset such that $\Lambda \subset A$. A monomial $m = \lambda_1 \cdots \lambda_k$ of elements in $\Lambda$ is called \textit{standard} if $\lambda_1 \leq \cdots \leq \lambda_k$. For any two monomials $n_1 = \alpha_1 \cdots \alpha_k$ and $n_2 = \beta_1 \cdots \beta_\ell$ with elements $\alpha_1, \ldots, \alpha_k, \beta_1, \ldots, \beta_\ell$ in $\Lambda$, we say $n_1 \leq n_2$ if either $\alpha_1 \cdots \alpha_k$ is an initial subsequence of $\beta_1 \cdots \beta_\ell$ or if $\alpha_i < \beta_i$ for the first $i$ where $\alpha_i \neq \beta_i$. A \textit{straightening law} on $\Lambda$ for $A$ is a set of distinct algebra generators $\{ \bar{\lambda} \mid \lambda \in \Lambda \}$ for $A$ over $A'$, such that any monomial $m = \bar{\lambda_1} \cdots \bar{\lambda_k}$ in $A$ can be uniquely expressed as an $A'$-linear combination of standard monomials $m_i$ of $A$ with $m_i \leq m$ \cite{straightening_algebra}.

Consider the notations as described above.

\begin{rem} \label{rem_weakdseq}
Huneke provided a class of weak $d$-sequences in \cite[Proposition 1.3]{CH1} arising from algebras with straightening laws. He showed that if $A$ is an ASL on $\Lambda$ over $A'$, and $\Sigma$ is a poset ideal of $\Lambda$ such that, for non-comparable elements $\alpha, \beta \in \Sigma$, if the straightening of $\bar{\alpha} \bar{\beta}$ is $\sum r_i \bar{\gamma_i} \bar{m_i}$, where $\gamma_i < \alpha$ and $\gamma_i < \beta$, then $\bar{m_i} \in \Sigma$, it follows that $\{ \bar{\alpha} \mid \alpha \in \Sigma \}$ forms a weak $d$-sequence.
\end{rem}

\begin{thm} \label{weak_d-seq}
If $A$ is an algebra with a straightening law (ASL) on $\Lambda$ over a commutative ring $A'$, and $\Sigma \subseteq \Lambda$, then $\{ \bar{\lambda} \mid \lambda \in \Sigma \}$ forms a $d$-sequence in $A$ if and only if $\{ \bar{\lambda} \mid \lambda \in \Sigma \}$ forms a weak $d$-sequence with respect to some total order on $\Lambda$.
\end{thm}

\begin{proof}

$\Rightarrow$\\
Clearly, a $d$-sequence is always a weak $d$-sequence with respect to a total order.\\

$\Leftarrow$\\
For the converse, assume $\{ \bar{\lambda} \mid \lambda \in \Sigma \}$ forms a weak $d$-sequence with respect to a total order. Then, for each $\lambda \in \Sigma$, the poset ideal such that $\lambda$ lies above it has the form $\Sigma' = \{ \alpha \mid \alpha < \lambda \}$. From the definition of a weak $d$-sequence, we have the following:
\begin{enumerate}
    \item $(\bar{I_{\Sigma'}} : \bar{\lambda}) \cap \bar{I_{\Sigma}} = (\bar{I_{\Sigma'}} : \bar{\lambda})^*$.
    \item If $\bar{\beta} \in (\bar{I_{\Sigma'}} : \bar{\lambda})$, then $\bar{\lambda} \bar{\beta} \in \bar{I_{\Sigma'}} \bar{I_{\Sigma}}$.
\end{enumerate}

To prove that $\{ \bar{\lambda} \mid \lambda \in \Sigma \}$ forms a $d$-sequence, it suffices to show that $(\bar{I_{\Sigma'}} : \bar{\lambda}) \cap \bar{I_{\Sigma}} = \bar{I_{\Sigma'}}$. Clearly, $(\bar{I_{\Sigma'}} : \bar{\lambda}) \cap \bar{I_{\Sigma}} \supseteq \bar{I_{\Sigma'}}$. 

Now, if $\bar{\beta} \in (\bar{I_{\Sigma'}} : \bar{\lambda})$, then from the definition of a weak $d$-sequence, we have $\bar{\lambda} \bar{\beta} \in \bar{I_{\Sigma'}} \bar{I_{\Sigma}}$, where $\bar{\lambda} \notin \bar{I_{\Sigma'}}$. This implies that $\bar{\lambda} \bar{\beta}$ can be expressed in the form $\bar{\gamma'} \bar{\gamma}$ where $\gamma' \in \Sigma'$ and $\gamma \in \Sigma$.

Since $\{ \bar{\lambda} \mid \lambda \in \Sigma \}$ forms a weak $d$-sequence with respect to a total order, every monomial is a standard monomial. In particular, $\bar{\lambda} \bar{\beta}$ satisfies this property. Since any monomial in $A$ can be uniquely written as an $A'$-linear combination of standard monomials of $A$, it will have a unique representation. Thus, we obtain $\bar{\beta} \in \bar{I_{\Sigma'}}$, which implies $(\bar{I_{\Sigma'}} : \bar{\lambda}) \cap \bar{I_{\Sigma}} = \bar{I_{\Sigma'}}$. Therefore, $\{ \bar{\lambda} \mid \lambda \in \Sigma \}$ forms a $d$-sequence.

\end{proof}


Consider the skew-symmetric matrix $X = \begin{bmatrix}
0      & x_{12} & \ldots & x_{1\;n} \\
-x_{12} &   0    & \ldots & x_{2\;n} \\
\vdots & \vdots & \ddots & \vdots \\
-x_{1\;n} & -x_{2\;n} & \ldots &   0   
\end{bmatrix} $ of odd order $n = 2r + 1$, where $r \in \mathbb{N} \cup \{0\}$, and the entries $x_{ij}$ for $i = 1, \ldots, n-1$ and $j = i + 1, \ldots, n$ are indeterminates. This matrix $X$ is called a generic skew-symmetric matrix of order $n$. In \cite{CH1} (Pg. 481), Huneke proved that the maximal order Pfaffians of $X$ form a weak $d$-sequence with respect to some term order and commented that it seems likely that they also form a $d$-sequence. In a personal communication, K. N. Raghavan asked for a proof of the $d$-sequence property of the maximal order Pfaffians, which motivated us to explore the connection between weak $d$-sequences and $d$-sequences.

\begin{cor} \label{max-pfaffians}
Let $X$ be a generic skew-symmetric matrix of odd order $n = 2r + 1$, where $r \in \mathbb{N}$. Then the maximal order Pfaffians form an unconditioned $d$-sequence.
\end{cor}

\begin{proof}
Let $X$ be a generic skew-symmetric matrix of odd order, and let $\left[ i_1, \ldots, i_{s} \right]$ represent the Pfaffian determined by the $i_1, \ldots, i_{s}$ rows and the corresponding columns of $X$. Consider a partial order $\leq$ on the set of Pfaffians as follows: $\left[ i_1, \ldots, i_{2k} \right] \leq \left[ j_1, \ldots, j_{2l} \right]$ if and only if $l \geq k$ and $i_m \geq j_m$ for $m = 1, \ldots, 2l$. This partial order defines a total order on the maximal order Pfaffians of $X$. According to Remark \ref{rem_weakdseq}, the maximal order Pfaffians form a weak $d$-sequence with respect to this partial order. Since the restriction of this partial order to the set of maximal order Pfaffians of $X$ becomes a total order, the $d$-sequence property follows from Theorem \ref{weak_d-seq}.

Furthermore, swapping rows and the corresponding columns of $X$ does not affect the total order with respect to which the maximal order Pfaffians form a weak $d$-sequence. Therefore, the maximal order Pfaffians form an unconditioned $d$-sequence.

\end{proof}

\subsection*{Acknowledgment}
We wish to express our profound gratitude to the late Prof. Jürgen Herzog for his time, expertise, and insightful suggestions. Our main result in Section \ref{sec2} is a direct consequence of a problem he posed to us during a personal communication.

The first author is supported by Core Research Grant (CRG/2023/007668), from Anusandhan National Research Foundation (ANRF), India. The second author is financially supported by the INSPIRE fellowship, Department of Science and Technology (DST), India.


\begin{thebibliography}{99}
\addcontentsline{toc}{chapter}{References}


\bibitem{AN2} M. Amalore Nambi, and N. Kumar: Regularity of powers of d-sequence (parity) binomial edge ideals of unicycle graphs. \emph{Comm. Algebra}, {\bf 52} (2024), no.~6, 2598–2615.

\bibitem{AN1} M. Amalore Nambi and N. Kumar: $d$-sequence edge binomials, and regularity of powers of binomial edge ideal of trees. \emph{J. Algebra Appl.} {\bf 23} (2024), no.~10, Article no. 2450154. 


\bibitem{Ali_Sara} A. Alilooee and S. Faridi: On the resolution of path ideals of cycles. \emph{Comm. Algebra} {\bf 43} (2015), no.~12, 5413-5433. 

\bibitem{Artin} E.~Artin: Geometric Algebra. \emph{Interscience Publishers, Inc., New York-London}, (1957), x+214.

\bibitem{Baetica} C. Baetica: Pfaffian ideals of linear type. \emph{Comm. Algebra}, {\bf 27} (1999), no.~8, 3909--3920.


\bibitem{Brumatti_Silva} P. Brumatti, A. F. da Silva: On the symmetric and {R}ees algebras of {$(n,k)$}-cyclic ideals.  \emph{Mat. Contemp.} {\bf 21} (2001), 27-42.

\bibitem{DBDE} D.~Buchsbaum, D.~Eisenbud: Algebra structures for finite free resolutions, and some structure theorems for ideals of codimension 3. \emph{Amer. J. Math.} {\bf 99} (1977), no.~3, 447-485.

\bibitem{CHV} A. Conca, J. Herzog, G. Valla: Sagbi bases with applications to blow-up algebras. \emph{J. Reine Angew. Math.}  { \bf 474} (1996), no. 1, 113-138.

\bibitem{CN} A.~Conca, E.~De Negri: {$M$}-sequences, graph ideals, and ladder ideals of linear type. \emph{J. Algebra}, {\bf 211}, no. 2, (1999), 299-624.

\bibitem{straightening_algebra} C. De Concini, C. Procesi: A characteristic-free approach to invariant theory. \emph{Advances in Math.} { \bf 3 } (1976), 330-354.

\bibitem{Hodge_algebras} C. De Concini, D. Eisenbud, C. Procesi: Hodge algebras. \emph{Soci\'{e}t\'{e} Math\'{e}matique de France, Paris} (1982).


\bibitem{Erey} N. Erey: Multigraded {B}etti numbers of some path ideals. Combinatorial structures in algebra and geometry \emph{Springer Proc. Math. Stat.}. {\bf 331} (2020), 51-65.



\bibitem{s-seq} J. Herzog, G. Restuccia, Z. Tang: s-Sequences and symmetric algebras. \emph{Manuscripta Math.} {\bf 104 } (2001), no.~4, 479-501.

\bibitem{HSV1} J. Herzog, A. Simis, W. V. Vasconcelos: Approximation complexes of blowing-up rings.  \emph{J. Algebra} {\bf 74} (1982), no.~2, 466-493.

\bibitem{CH2} {C.~Huneke}: On the symmetric and Rees algebra of an ideal generated by a d-sequence.  \emph{J. Algebra},  {\bf 62} (1980), no.~2, 268-275.


\bibitem{CH4} {C.~Huneke}: Powers of ideals generated by weak {$d$}-sequences.  \emph{J. Algebra},  {\bf 68} (1981), no.~2, 471-509.


\bibitem{CH1} {C.~Huneke}: The theory of d-sequences and powers of ideals,
\emph{Adv. Math.} {\bf 46} (1982), no~3, 249-279.


\bibitem{c-seq} H. Kulosman: Monomial sequences of linear type \emph{Illinois Journal of Mathematics} {\bf 52 } (2008), no.~4, 1213-1221.

\bibitem{Rom1} T. R\"omer: Homological properties of bigraded algebras. \emph{Illinois J. Math} {\bf 45} (2001), no.~4, 1361–1376.

\bibitem{monomial_seq} Z. Tang: On certain monomial sequences. \emph{J. Algebra} {\bf 282 } (2004), no.~2, 831-842.

\bibitem{Valla} G. Valla: On the symmetric and {R}ees algebras of an ideal. \emph{Manuscripta Math.} {\bf 30} (1980), no.~3, 239–255.



\end{thebibliography}
\end{document}